\documentclass[a4paper,12pt]{article}
\usepackage[a4paper, margin=3.5cm, top=4cm]{geometry}
\usepackage{amssymb,amsmath,amsthm,thmtools,upref}
\usepackage[dvipsnames]{xcolor}
\usepackage[pdftex]{graphicx}
\usepackage{mathabx}
\usepackage{tikz}
\usepackage{wasysym}
\usepackage[all]{xy}
\usepackage{relsize}
\usepackage{pgf}

\usepackage{hyperref}

\usetikzlibrary{positioning}
\usetikzlibrary{cd}
\usetikzlibrary{intersections}

\newtheorem{theorem}{Theorem}[section]
\newtheorem{lemma}[theorem]{Lemma}
\newtheorem{proposition}[theorem]{Proposition}
\newtheorem{corollary}[theorem]{Corollary}

\theoremstyle{coloreddef}
\newtheorem{definition}[theorem]{Definition}
\newtheorem{proposition-definition}[theorem]{Proposition-Definition}

\theoremstyle{remark}
\newtheorem*{remark}{Remark}
\newtheorem*{example}{Example}
\newtheorem*{examples}{Examples}

\theoremstyle{problem}

\newcommand{\oo}{\mathcal{O}}

\newcommand{\mc}{\mathcal}
\newcommand{\mbb}{\mathbb}

\DeclareMathOperator{\GL}{GL}

\DeclareMathOperator{\lv}{lv}

\definecolor{mypink1}{RGB}{231, 62, 156}
\definecolor{mypink2}{RGB}{219, 48, 122}
\definecolor{mygreen}{RGB}{107, 165, 51}
\definecolor{mygreen2}{RGB}{16,131,58}
\definecolor{teal}{RGB}{72, 146, 182}
\definecolor{lightpurple} {RGB}{172, 82, 204}
\definecolor{darkpurple}{RGB}{103,17,168}
\definecolor{myred1}{RGB}{128,0,32}
\definecolor{myblue1}{RGB}{24,189,225}
\definecolor{myblue2}{RGB}{16,25,144}

\begin{document}

\title{An approach to harmonic analysis on non-locally compact groups II: an invariant measure on groups of ordered type} \date{} \author{Raven Waller\thanks{This work was completed while the author was supported by an EPSRC Doctoral Training Grant at the University of Nottingham.}}

\maketitle

\begin{abstract}
We consider a class of non-locally compact groups on which one may define a left-invariant, finitely additive measure taking values in some finitely generated extension of the field $\mbb{R}$ of real numbers. In particular, we recover previously studied special cases, along with the case of reductive algebraic groups defined over higher dimensional local fields.
\end{abstract}

\section{Introduction} \label{sec:introduction}
The existence of a real-valued Haar measure $\mu$ on a topological group $G$ is essentially equivalent to $G$ being locally compact. If one relaxes the condition that $\mu$ take values in $\mbb{R}$, one may define a measure on slightly more general groups. For example, Fesenko (\cite{fesenko-aoas1}) defined a finitely-additive invariant measure on the additive group of a two-dimensional local field $F$ taking values in the two-dimensional field $\mbb{R} (\!( X )\!)$ of formal Laurent series with real coefficients.

Fesenko's construction has many influences from non-standard analysis. In particular, the indeterminate $X$ is considered to be an infinitesimal positive element, so that $X>0$, and for any positive real number $a$ and positive integer $n$ we have $X^n < aX^{n-1}$ (and so in particular $X$ is smaller than any positive real number). Similar ideas were behind the introduction of the concept of a level structure over a locally compact group $X$, which was studied by the author in \cite{waller-level} in part in order to further generalise Fesenko's measure.

While there have been several generalisations of Fesenko's construction (for example, by Morrow in \cite{morrow-2dint} and \cite{morrow-gln}, by van Urk in \cite{wester-fourier}, and by the author in \cite{waller-GL2}), such constructions have generally only appeared for specific examples, and any hint of a general theory has been (until now) remarkably absent.

Due to the author's background in studying higher dimensional local fields, the initial search for the general theory focused on the idea of a ``valuation space". However, it quickly became apparent that the presence of a valuation was not such an important aspect in the existing examples - rather, it was certain topological properties that just happened to be a consenquence of the existence of a valuation. The development of this idea led to the notion of level structure - a topological framework in which higher local fields and the intermediary ``valuation spaces" existed, but were no longer the primary focus (though still serving as important examples).

Once one has a group equipped with a level structure, one can define the notion of level-compactness. A weaker version of compactness that works ``on the level", this appears to be enough to get around many obstacles which appear due to the lack of compactness of groups related to higher dimensional number theory.

Unfortunately, there are still several other obstacles which make it difficult to generalise analytic proofs using level structures. For example, it is not true that every subset of $\mbb{R}(\!(X)\!)$ which is bounded above (with respect to the natural ordering on this field considering $X$ as an infinitesimal) has a supremum, and a sequence which is bounded above by some constant may have a limit which exceeds this bound. This make it difficult to generalise, for instance, the Riesz Representation Theorem, which is a fundamental result in real-valued measure theory.

In this paper, we thus work with a slightly narrower class of spaces which retain enough similarity with higher local fields that the existing constructions of \cite{fesenko-aoas1} and \cite{waller-GL2} can be applied. Although this means sacrificing a fair amount of generality, it does leave us in the position to give some of the first general statements regarding ``higher dimensional" measure. As an example, we have the following.

\begin{theorem}
Let $R$ be a Henselian ring with locally compact residue field $k$, and let $G$ be an algebraic group defined over $R$. If there exists a ``nice" level structure $\mc{L}_G$ for $G(R)$ over $G(k)$ of some elevation $e \ge 0$, then there is a left-invariant, finitely additive measure on $G(R)$ which takes values in $\mbb{R} (\!( X_1 )\!) \cdots (\!( X_e)\!)$ and is sufficiently compatible with the Haar measure on $G(k)$.
\end{theorem}

Here ``nice" means that $G(R)$ is rigid, properly suspended, and of ordered type with respect to this level structure (all of these terms will be explained later in the text), and satisfies certain compatibility relations with $G(k)$.

The contents of this paper are as follows. We begin in Section \ref{sec:level} by reviewing the important definitions of level structure and level-compactness from \cite{waller-level}. In Section \ref{sec:orderedtype} we then give the definition of groups of ordered type, which will be the main object of study in this text. We also describe how the level structure behaves in this case with respect to the ring of ddd-sets, which will be the (minimal) ring on which our invariant measure will be defined.

In Section \ref{sec:orderedmeasure} we show that the results of \cite{waller-GL2} may be sufficiently generalised to the case of groups of ordered type. We then present some ``compatibility conditions" under which one may define a left-invariant, finitely additive measure on such a group $G$ taking values in some finitely generated field extension of $\mbb{R}$, such that this measure is an appropriate ``lift" of the Haar measure on the base $X$ over which the level structure is defined. In particular, we recover all of the special cases which have been studied previously.

Finally, in Section \ref{sec:induced} we use the induced level structure to obtain a much richer selection of examples than has previously been available. Most importantly, this includes the case of reductive groups over higher dimensional local fields. We highlight this example in particular due to the importance of such groups in representation theory. (Previously, only the particular case of $\GL_n(F)$ had been studied in any detail.)

\textbf{Notation.} If $F$ is an $n$-dimensional local field, we will denote by $\oo_F$ the rank one ring of integers of $F$, and by $O_F$ the rank $n$ ring of integers. We label the successive residue fields so that $F_{k-1} = \overline{F}_k$, so that (for example) $\overline{F} = F_{n-1}$, and $F_1$ is a local field. We also fix a system of local parameters $t_1, \dots, t_n$ for $F$ so that (for instance) $t_n$ is a uniformiser of the ring $\oo_F$.

For a nonnegative integer $e$ we always order the group $\mbb{Z}^e$ lexicographically from the right, where we take $\mbb{Z}^0 = \{ 0 \}$ by convention.

\textbf{Acknowledgements.} I would like to thank Ivan Fesenko for his many comments and suggestions throughout the writing of this text. I would also like to thank Kyu-Hwan Lee, Sergey Oblezin, and Tom Oliver, with whom I have discussed several aspects of this work.

\section{Level structures and level-compactness} \label{sec:level}
We recall from \cite{waller-level} the required definitions of level structure and level-compactness.

\begin{definition} \label{levelstructuredef}
Let $X$ be a locally compact topological group, and let $e \ge 0$ be an integer. A group $G$ is levelled over $X$ (with elevation $e$) if there is a collection $\mc{L}$ of subsets of $G$ satisfying the following conditions:

(1) Each element of $\mc{L}$ contains the identity element $e_G$ of $G$.

(2) $\mc{L}$ indexed by $\mc{U}(1) \times \mbb{Z}^e$, where $\mc{U}(1)$ is a basis of neighbourhoods of the identity in $X$ and $\mbb{Z}^e$ is lexicographically ordered from the right.

(3) For any $U, V \in \mc{U}(1)$ with $V \subset U$, if $G_{V,\gamma}, G_{U,\delta} \in \mc{L}$ with $\gamma \le \delta$ then $G_{V,\gamma} \cap G_{U,\delta} = G_{V,\delta}$.

(4) For any fixed $\gamma \in \mbb{Z}^e$, $G_{U,\gamma} \cup G_{V,\gamma} = G_{U \cup V, \gamma}$ and $G_{U,\gamma} \cap G_{V,\gamma} = G_{U \cap V, \gamma}$.

The collection $\mc{L}$ is called a level structure.
\end{definition}

\begin{examples}
(1) Any locally compact group $G$ is levelled over itself with elevation $0$. In this case $\mc{L} = \mc{U}(1)$.

(2) Let $F$ be an $n$-dimensional local field (which may be archimedean) and $F_1$ is its $(n-1)^{st}$ residue field. Then $F$ is levelled over $F_1$ with elevation $n-1$. If $F$ and $F_1$ have the same characteristic, so that $F$ is isomorphic to $F_1 (\!(t_2)\!) \dots (\!(t_n)\!)$, $\mc{L}$ consists of sets of the form $$t_2^{i_2} \dots t_n^{i_n}B(0,r) + \sum_{j=2}^n t_j^{i_j +1} t_{j+1}^{i_{j+1}} \dots t_n^{i_n} F_1 (\!(t_2)\!) \dots (\!(t_{j-1})\!)[\![ t_j]\!],$$ where $B(0,r)$ is the open ball of radius $r$ in $F_1$. In the mixed characteristic case, we associate to the pair $( t_1^{i_1} \oo_{F_1} , (i_2, \dots, i_n) )  \in \mc{U}(1) \times \mbb{Z}^{n-1}$ the set $t_1^{i_1} \dots t_n^{i_n} O_F \subset F$.

(3) Since the form taken by elements of $\mc{L}$ in the previous example may look quite complicated, we give a concrete example in dimension 3 to illustrate the general phenomenon. Let $F = \mbb{F}_p (\!(t_1)\!) (\!(t_2)\!) (\!(t_3)\!) $, so that $F_1 = \mbb{F}_p (\!(t_1)\!)$. The open balls in $F_1$ are then simply the fractional ideals $t_1^{i_1} \mbb{F}_p [\![t_1]\!]$ for $i_1 \in \mbb{Z}$. Elements of $\mc{L}$ are thus of the form $$t_1^{i_1} t_2^{i_2}t_3^{i_3} \mbb{F}_p [\![t_1]\!] + t_2^{i_2 +1} t_3^{i_3} \mbb{F}_p (\!(t_1)\!) [\![ t_2 ]\!] + t_3^{i_3+1} \mbb{F}_p (\!(t_1)\!) (\!( t_2 )\!) [\![ t_3]\!].$$ Note that this is exactly the set $t_1^{i_1}t_2^{i_2}t_3^{i_3} O_F$.

(4) For $F$ a two-dimensional nonarchimedean local field, the groups $$K_{i,j} = I_2 + t_1^i t_2^j M_2(O_F)$$ for $i,j > 0$ define a partial level structure on the group $\GL_2(F)$ over $\GL_2 (\overline{F})$. This particular case is discussed in much more detail in \cite{waller-GL2}.
\end{examples}

\begin{remark}
It was noted in \cite{waller-level} that the definition of a level structure does not necessarily require that the base $X$ be locally compact, and in fact this condition may be replaced with something much weaker. However, since the goal of the current text is to define an invariant measure, here it is extremely important that we work over a locally compact base.
\end{remark}

\begin{definition}
We equip $G$ with the level topology as follows. We take $\mc{L}$ as a basis of neighbourhoods of the identity, and then extend to other points of $G$ by insisting that multiplication by any fixed element be continuous.
\end{definition}

If $G$, is a locally compact group viewed as being levelled over itself with $e=0$, this is just the original topology on $G$. On the other hand, if $G$ is (the additive group of) a two-dimensional local field as in the second example, the level topology on $G$ is \textit{not} the usual two-dimensional topology as defined (for example) in \cite{madunts-zhukov} - in this topology elements of $\mc{L}$ are closed but not open, for example. 

\begin{definition}
An element of $G \mc{L} = \{ g H : g \in G, H \in \mc{L} \}$ is called a distinguished set. We also allow the empty set to be distinguished.
\end{definition}

\begin{example}
The distinguished subsets of a two-dimensional local field $F$ as defined by Fesenko in \cite{fesenko-aoas1} are exactly the distinguished sets of the elevation $1$ level structure of $F$ over its residue field, namely those of the form $\alpha + t_1^it_2^jO_F$. Similarly, the distinguished sets of $\GL_2(F)$ in \cite{waller-GL2} are precisely those of the form $gK_{i,j}$.
\end{example}

\begin{definition}
The ring of ddd-sets of $G$ with respect to the level structure $\mc{L}$ is the minimal ring of sets containing $G \mc{L}$.
\end{definition}

\begin{remark}
By construction, for a two-dimensional local field $F$, the ddd-sets of $F$ and $\GL_2(F)$ are exactly the ones defined in \cite{fesenko-aoas1} and \cite{waller-GL2}. As in these specific cases, the ring of ddd-sets provides a natural structure on which to define an invariant measure.
\end{remark}

\begin{definition} \label{leveldef}
For $G_{U,\gamma} \in \mc{L}$, we define its level $$\lv(G_{U,\gamma}) = \max \{ \delta \in \mbb{Z}^e : G_{U,\gamma} \subset G_{U,\delta}  \}.$$ We then put $\lv(gG_{U,\gamma}) = \lv(G_{U,\gamma})$ for any $g \in G$.  For a general subset $S \subset G$, the level of $S$ (if it exists) is the minimal level of any subset of $S$ of the form $g G_{U,\gamma}$ for $g \in G$, $U \in \mc{U}(1)$, $\gamma \in \mbb{Z}^e$. We write $\lv(S)$ for the level of $S$.
\end{definition}

Clearly we have the equality $$\lv(S) = \min \{ \lv(S') : S' \subset S \text{ has a level} \}.$$

Since the level of a subset is far from being an ``everywhere local" property, in order to consider compactness we require the following more refined notion of uniform level.

\begin{definition} \label{uniformleveldef}
A subset $S \subset G$ has uniform level $\gamma$ if $\lv(S) = \gamma$ and for every point $s \in S$ there is a distinguished set $D_s$ of level $\gamma$ with $s \in D_s$ and $D_s \subset S$.
\end{definition}

\begin{definition} \label{levelcompact}
Let $G$ be a group with level structure, and let $\gamma \in \mbb{Z}^e$. A subset $S \subset G$ is called $\gamma$-compact if every open cover (in the level topology) of $S$ by sets of uniform level $\gamma$ has a finite subcover. We will call $S$ level-compact if there is some $\gamma \in \mbb{Z}^e$ such that $S$ is $\gamma$-compact. More generally, $S$ is called locally level-compact if for every $s \in S$ there is some $\gamma \in \mbb{Z}^e$ such that $s$ has a $\gamma$-compact neighbourhood.
\end{definition}

\begin{remark}
As was hinted previously, it is important that each set in the cover has \textit{uniform} level $\gamma$. (See \cite{waller-level} for examples of what can go wrong if this is not the case.) Note that although we refer to open covers in the definition (so that the reader may immediately see the connection with compactness), we may in fact omit the word "open" since any set of uniform level is necessarily open.
\end{remark}

The standard example to keep in mind is the following.

\begin{proposition} \label{Fcompactness}
Let $F$ be a $d$-dimensional nonarchimedean local field with parameters $t_1, \dots, t_d$. If $F$ is given the level structure of elevation $d-1$ over it's $1$-dimensional residue field then the subset $t_1^{i_1} \cdots t_d^{i_d} O_F$ is $\gamma$-compact with $\gamma = (i_2, \dots, i_d)$.
\end{proposition}

\begin{proof}
See \cite{waller-level}.
\end{proof}

\begin{remark}
It is essential that the elements of the cover all have the same level as $t_1^{i_1} \cdots t_d^{i_d} O_F$ in this Proposition. Indeed, $O_F = \bigcup \alpha + t_2O_F$ where $\alpha$ runs through an (infinite!) set of representatives for $O_F / t_2 O_F$, and since the union is disjoint there can be no finite subcover. It is equally important that the elements of the cover have uniform level, as we saw in an earlier example.
\end{remark}

\begin{remark}
The proof of Proposition \ref{Fcompactness} uses the fact that $F$ is complete in an essential way. As some of the consequences of completeness will be crucial later, it is worthwhile to ask if the completeness property (or perhaps a weaker alternative which still works for this proof) can be restated purely in terms of the level structure.
\end{remark}

In the general case, it will also be important to impose the following condition, in order to eliminate various pathological cases.

\begin{definition}
A group $G$ levelled over $X$ with elevation $e$ is rigid if it satisfies the following condition: for any $\gamma \in \mbb{Z}^e$, if $G$ contains at least one subset of level $\gamma$ then $\lv(G_{U,\gamma}) = \gamma$ for all $U \in \mc{U}(1)$.
\end{definition}

\section{Groups of ordered type} \label{sec:orderedtype}

In this section we restrict attention to a particular kind of level structure which exhibits the same desirable properties as higher dimensional nonarchimedean fields. While it is unfortunate that this class of spaces does not include the archimedean case, this restriction makes so many things simpler that it is well worth studying as a special case. 

We begin with the definition, which one may recognise from the statement of Lemma 3.3 in \cite{waller-GL2}. Given how powerful this result was in that paper, it is not surprising that we would want to consider how it can be applied in the more general setting.

\begin{definition}
A group $G$ levelled over $X$ is of ordered type if $D_1 \cap D_2 \in \{ D_1, D_2, \emptyset \}$ for any two distinguished subsets $D_1, D_2$ of $G$.
\end{definition}

\begin{remark}
The terminology comes from the following observation: $G$ is of ordered type if and only if every family of pairwise nondisjoint distinguished subsets of $G$ is totally ordered with respect to inclusion.
\end{remark}

\begin{example}
Nonarchimedean higher local fields $F$ are of ordered type (see \cite{fesenko-aoas1}), as are the groups $\GL_2(F)$ (see \cite{waller-GL2}). Archimedean fields are not: $(0,2) + t \mbb{R}[\![ t]\!]$ and $(1,3) + t \mbb{R}[\![ t]\!]$ are distinguished subsets of $\mbb{R}(\!(t)\!)$ with nonempty intersection but neither is contained inside the other.
\end{example}

In \cite{waller-level}, we saw that the level operation behaves reasonably well under intersections, but noted that it is very badly behaved with respect to union. For groups of ordered type, however, the level is much more controllable.

\begin{proposition} \label{orderedunionlevel}
Let $G$ levelled over $X$ be of ordered type, and let $A,B$ be a distinguished subsets of $G$ with levels $\gamma, \delta$ respectively. If $A \cup B$ is not a distinguished set then it has level $\min \{ \gamma, \delta \}$.
\end{proposition}

\begin{proof}
Let $\beta = \min \{ \gamma, \delta \}$. Then certainly $A \cup B$ contains a distinguished set of level $\beta$, and so if it has a level then the level must be at most $\beta$. Suppose that $A \cup B$ contains a distinguished set $D$ of level $< \beta$.

Now, since they are distinguished, $D \cap A$ is either empty or equal to $A$ (since $A$ has the higher level), and similarly $D \cap B$ is either empty or equal to $B$. Thus $D = D \cap (A \cup B) = (D \cap A) \cup (D \cap B) \in \{ \emptyset, A, B, A\cup B \}$. But since $D$ is distinguished it cannot be empty, and since $D$ has lower level than both $A$ and $B$ it cannot be equal to either of them. Thus the only possibility is that $D = A \cup B$ is distinguished, which is not true by assumption.

We have thus shown that $A \cup B$ contains no distinguished set of level lower than $\beta$, but since it contains both $A$ and $B$ it must contain a distinguished set of level $\beta$, and so it follows that $A \cup B$ has level $\beta$.
\end{proof}

\begin{remark}
If we remove the assumption that $A \cup B$ not be distinguished then the level may decrease, and it is easy to construct examples to show that this decrease can be arbitrarily large. However, $A \cup B$ still has a level in this case, since it is distinguished, and so we see that any union of two distinguished sets has a level. By imposing certain extra conditions one may obtain a bound on the possible decrease in level.
\end{remark}

\begin{corollary}
If $G$ is of ordered type and $A_1, \dots, A_r$ are finitely many distinguished subsets of $G$ such that no union of any collection of the $A_i$ is a distinguished set, then $A = \bigcup_{i=1}^r A_i$ has level $\min \{ \lv(A_i) \}$.
\end{corollary}

\begin{proof}
Let $\beta = \min \{ \lv(A_i) \}$ and suppose that $A$ contains a distinguished set $D$ of level $< \beta$. Then as in the proof of the Proposition we must have $D \cap A_i \in \{ \emptyset, A_i \}$ for every $i$, hence $D=D \cap A$ is either empty or equal to a union of some number $\ge 2$ of $A_i$, which is a contradiction since neither of these are distinguished. We then conclude as before that $\lv A = \beta$.
\end{proof}

\begin{corollary} \label{orderedunionlevel2}
If $G$ is of ordered type and $A_1, \dots, A_r$ are finitely many distinguished sets then $A = \bigcup_{i=1}^r A_i$ has a level.
\end{corollary}

\begin{proof}
If no subunion of $A = \bigcup_{i=1}^r A_i$ is equal to a distinguished set then we use the previous Corollary. If there is a subunion $\bigcup_{j} A_{i_j}$ which is equal to a distinguished set $B$, we replace the subcollection $\{ A_{i_j} \}$ with the set $B$. Since there are only finitely many $A_i$ to begin with, after at most finitely many replacements we will be in the first case.
\end{proof}

Unfortunately, the notion of having uniform level is particularly badly behaved with respect to unions, even when $G$ is of ordered type. (Indeed, the trouble that unions can cause were the motivation for the definition.) In general it is not even true that the union of two sets of the same uniform level $\gamma$ has uniform level.

\begin{example}
Let $F = \mbb{F}_p(\!(t_1)\!)(\!(t_2)\!)$, and $O_F = \mbb{F}_p[\![t_1]\!] + t_2 \mbb{F}_p(\!(t_1)\!) [\![ t_2]\!]$ be the rank two ring of integers of $F$. Put $A = t_2O_F \cup \left( t_2^{-1} + t_2O_F \right)$ and $B = O_F \backslash t_2O_F$. Both $A$ and $B$ have uniform level $1$, but $A \cup B = O_F \cup \left( t_2^{-1} + t_2O_F \right)$ does not have uniform level. (It has level $0$ since it contains $O_F$, but the point $t^{-1}$ has no distinguished neighbourhood of level $0$.)
\end{example}

Let us now consider what happens with differences.

\begin{proposition} \label{difflevel}
Let $G$ be Hausdorff and of ordered type. Then any distinguished subset of $G$ is closed, and if $A$ and $B$ are distinguished subsets of $G$ with $B \subset A$ such that $\lv(B) > \lv(A)$ and $A \backslash B$ has a level then $\lv(A \backslash B) \ge \lv(A).$
\end{proposition}

\begin{proof}
Clearly $A \backslash B$ cannot contain any distinguished set of level $< \lv(A)$, and so it is enough to show that it contains a distinguished set of some level.

Let $b \in B$, $x \in A \backslash B$. Then $B_x = xb^{-1} B$ is a distinguished subset of $G$, and since it contains $x=xb^{-1}b$ we have $B_x \cap A \neq \emptyset$. We thus have $B_x \cap A \in \{B_x, A \}$, and since $\lv(B_x) = \lv(B) > \lv(A)$, from \cite{waller-level} we have $\lv(B_x \cap A) \ge \lv(B_x) > \lv(A)$, hence we must have $B_x \cap A = B_x$, i.e. $B_x \subset A$.

Now, since $G$ is Hausdorff we can find disjoint open subsets $U_x, U_b \subset G$ with $x \in U_x$ and $b \in U_b$. Since $B_x$ and $B$ are both open sets we may assume that $U_x \subset B_x$ and $U_b \subset B$. Furthermore, since open sets are unions of distinguished sets, we may assume that $U_x$ and $U_b$ are in fact distinguished sets. But $x \in U_x \backslash B$ and $b \in B \backslash U_x$ implies $U_x \cap B = \emptyset$, hence $U_x$ is a distinguished subset of $A \backslash B$.

It remains to show that distinguished sets are closed. Let $D$ be any distinguished subset of $G$. By the same argument as in the previous paragraph, for any $x \in G \backslash D$ we can find a distinguished set $U_x$ which contains $x$ and is disjoint from $D$. But then $G \backslash B = \bigcup_{x \in G \backslash B} U_x$ is an open set, hence $D$ is closed.
\end{proof}

The following Proposition is not important for the remainder of the current text, but we present it here as an example of the kind of conditions one can impose in order to get a bound on the level. In particular, this result holds for higher dimensional local fields (although in that case one can easily compute the level directly).

\begin{proposition} \label{subgroupdistlevel}
Let $G$ be Hausdorff and of ordered type, and suppose that each element of $\mc{L}$ is a subgroup of $G$. Let $A$ be a distinguished subset of $G$ of level $\gamma$ and let $B \subset A$ be a distinguished subset of $G$ of level $\delta > \gamma$. If $A \backslash B$ has a level then $\gamma \le \lv(A \backslash B) \le \delta$.
\end{proposition}

\begin{proof}
The first inequality was shown in the previous Proposition, and so it remains to show that $A \backslash B$ contains a distinguished set of level $\delta$. Write $A=g G_{U,\gamma}$, $B = h G_{V,\delta}$ with $g,h \in G$ and $G_{U,\gamma}, G_{V,\delta} \in \mc{L}$. Then $g^{-1}h G_{V,\delta}$ is a coset of $G_{V,\delta}$ in $G_{U,\gamma}$. Since we assumed that $B$ has strictly larger level than $A$, $A \backslash B$ is nonempty, hence there exists at least one other coset $k G_{V,\delta}$. Thus $gk G_{V,\delta} \subset A$ and $B \cap gk G_{V,\delta} = g(g^{-1}h G_{V,\delta} \cap k G_{V,\delta}) = \emptyset$, hence $A \backslash B$ contains the distinguished set $gkG_{V,\delta}$ of level $\delta$.
\end{proof}

We would now like to combine some of our results in order to get something for ddd-sets. To do this, we must first recall the following classification of subsets of $G$ with no level.

\begin{definition}
Let $A$ be a subset of $G$ such that $\lv(S)$ does not exist. If $A$ contains no distinguished set, then $A$ is of type $S$. If the set $\{ \lv(D) : D \subset A \text{ is distinguished} \}$ is not bounded below (in other words if $A$ contains distinguished sets of arbitrarily low level), then $A$ is of type $L$. Otherwise, $A$ is of type $E$.
\end{definition}

\begin{remark}
If the elevation $e \le 1$, there can be no subsets of type $E$.
\end{remark}

\begin{proposition}
Let $G$ be Hausdorff and of ordered type, and for $1 \le i \le r$ let $A_i, B_j$ be distinguished (or empty) subsets of $G$ with each $B_j \subset A_i$ for some $i$ and the sets $A_i \backslash \bigcup_j B_j$ pairwise disjoint. If no union of any collection of $A_i \backslash \bigcup_j B_j$ is a distinguished set and $C = \bigcup_i A_i \backslash \bigcup_j B_j$ is not of type $E$ then $C$ has level at least $\min \{ \lv (A_i) \}$.
\end{proposition}

\begin{proof}
Let $\gamma = \min \{ \lv(A_i) \}$, $\delta = \min \{ \lv(A_i \backslash \bigcup_j B_j ) \}$. By construction $C$ contains a distinguished set of level $\delta$, and so all we must show is that it contains no distinguished subset of level lower than $\gamma$.

Suppose $C$ contains a distinguished set $D$ of level $< \gamma$. First of all this means that $\bigcup_i A_i$ contains $D$. Since they are distinguished we have $A_i \cap D \in \{ A_i, D, \emptyset \}$, and since $D$ has lower level than all of the $A_i$ we in fact have $A_i \cap D \in \{ A_i, \emptyset \}$, and so $(A_i \backslash \bigcup_j B_j) \cap D \in \{ (A_i \backslash \bigcup_j B_j), \emptyset \}$. But $D = D \cap C$ cannot be empty or equal to a single $A_i \backslash \bigcup_j B_j$, hence it is a union of two or more of the $A_i \backslash \bigcup_j B_j$, which violates our assumption that no subunion inside $C$ be distinguished.
\end{proof}

\begin{corollary}
If $G$ is Hausdorff and of ordered type then every ddd-set of $G$ either has a level or is of type $E$.
\end{corollary}

\begin{proof}
By definition any ddd-set can be written as a difference of finite disjoint unions $C = \bigcup_i A_i \backslash \bigcup_j B_j$ with each $B_j$ contained in some $A_i$. If any subunion $\bigcup_k A_{i_k} \backslash \bigcup_j B_{j}$ inside $C$ is equal to a distinguished set $D$, replace the collection $\bigcup_k A_{i_k} \backslash \bigcup_j B_{j}$ with the set $D$ (which has a level since it is distinguished). Since we begin with only finitely many sets, after finitely many such replacements we will be in the situation of the previous Proposition.
\end{proof}

The above result may still hold even if $G$ is not of ordered type. (It may be proven directly for an archimedean higher local field, for example.)

\section{An invariant measure on groups of ordered type} \label{sec:orderedmeasure}

In \cite{fesenko-aoas1} and \cite{waller-GL2}, one of the most crucial results that facilitated the construction of an invariant measure on $F$ and $\GL_2(F)$ for a two-dimensional local field $F$ was the property that the intersection of two distinguished sets is either empty or equal to one of them. Since this intersection property is the very definition of a group $G$ being of ordered type, we are naturally inclined to search for an invariant measure on such spaces in general.

The constructions of \cite{waller-GL2} were in fact written with this more general setting in mind, and so while we reformulate the important parts of this in the more general setting of groups of ordered type, the groups $F$ and $\GL_2(F)$ are good motivating examples to keep in mind throughout this section.

First of all, we must extend the notion of index of a subgroup to distinguished sets. The following definitions make sense for all groups with level structure, not just those of ordered type.

\begin{definition}
Let $D_1$ and $D_2$ be distinguished subsets of $G$ such that $D_1 \subset D_2$. We say that $D_1$ has finite index in $D_2$ if there are finitely many $g_i \in G$ such that $\bigcup_i g_i D_1$ is an open cover of $D_2$.
\end{definition}

\begin{definition}
Let $D_1$ and $D_2$ be distinguished subsets of $G$ and suppose that $D_1$ has finite index in $D_2$. Then the index $|D_2:D_1|$ of $D_1$ in $D_2$ is defined to be the minimum cardinality $$\min \left\{ \# I : \# I < \infty, D_2 \subset \bigcup_{i \in I} g_i D_1 \right\} $$ of all finite open covers of $D_2$ by translates of $D_1$.
\end{definition}

\begin{remark}
The set $\left\{ \# I : \# I < \infty, D_2 \subset \bigcup_{i \in I} g_i D_1 \right\}$ is a nonempty subset of $\mbb{Z}$ and is bounded below by $1$, and so this minimum always exists. If $D_1$ is a subgroup of $D_2$, this coincides with the usual definition of index, and it also coincides with the use of this terminology in \cite{waller-level} for the group $G=\GL_2(F)$ with $F$ a two-dimensional local field. 
\end{remark}

By abuse of notation we will write $|D_2 :D_1| < \infty$ if $D_1$ has finite index in $D_2$.

\begin{lemma} \label{indextowerlaw}
Let $D_1, D_2,$ and $D_3$ be distinguished subsets of $G$ with $D_1 \subset D_2 \subset D_3$. Then $|D_3 : D_1|$ is finite if and only if both $|D_3:D_2|$ and $|D_2 : D_1|$ are finite, in which case $|D_3 : D_1| \le |D_3:D_2| \cdot |D_2 : D_1|$.
\end{lemma}

\begin{proof}
First suppose $|D_3:D_2|$ and $|D_2 : D_1|$ are finite, so that $D_3 \subset \bigcup_{i=1}^m g_i D_2$ and $D_2 \subset \bigcup_{j=1}^n h_j D_1$ for some $g_i, h_j \in G$. This gives $D_3 \subset \bigcup_{i=1}^m \bigcup_{j=1}^n g_i h_j D_1$, and so $D_1$ has finite index in $D_3$, and taking $m$ and $n$ to be minimal gives the required inequality of indices.

Conversely, suppose that $|D_3:D_1|$ is finite, so that $D_3 \subset \bigcup_{k=1}^\ell f_i D_1$. Since $D_2 \subset D_3$, $\bigcup_{k=1}^\ell f_i D_1$ is already a finite cover of $D_2$, and so it follows immediately that $|D_2:D_1|$ is finite. On the other hand, since $D_1 \subset D_2$ we have $\bigcup_{k=1}^\ell f_i D_2 \supset \bigcup_{k=1}^\ell f_i D_1 \supset D_3$, and so $|D_3 : D_2|$ is also finite.
\end{proof}

\begin{corollary}  \label{strictindextowerlaw}
Let $G$ be of ordered type and let $D_1 \subset D_2 \subset D_3$ be distinguished subsets of $G$. If $|D_3:D_1|$ is finite then so are both $|D_3:D_2|$ and $|D_2 : D_1|$, and $|D_3 : D_1| = |D_3:D_2| \cdot |D_2 : D_1|$
\end{corollary}

\begin{proof}
Since $G$ is of ordered type we may in fact arrange that $D_3 = \bigcup_{i=1}^m g_i D_2$ is a disjoint union of translates of $D_2$ and $D_2 = \bigcup_{j=1}^n h_j D_1$ is a disjoint union of translates of $D_1$. This gives a disjoint union $D_3 = \bigcup_{i=1}^m \bigcup_{j=1}^n g_i h_j D_1$, which shows that at least $mn$ translates of $D_1$ are required to cover $D_3$.
\end{proof}

\begin{remark}
If $G$ is of ordered type, any two translates of the same distinguished set $D$ are either disjoint or equal, and so in this case distinguished subsets behave exactly like cosets. In particular, if we have two distinguished subsets $D_1 \subset D_2$, we can take a system of representatives $\{ g_i \} \subset G$ such that we have a disjoint union $D_2 = \bigcup_i g_i D_1$. 
\end{remark}

If $F$ is a two-dimensional local field with rank two ring of integers $O_F$, the index $|t_1^{i_1} t_2^{j_1} O_F : t_1^{i_2} t_2^{j_2} O_F |$ is finite if and only if $j_1 = j_2$. Furthermore, this fact plays a significant part in the definition of the measure on $F$, and so if we wish to extend such constructions to a more general setting it makes sense to restrict to examples where similar results hold.

\begin{definition}
Let $G$ be levelled over $X$. We say that $G$ is properly suspended if it satisfies the following property: for any two distinguished subsets $D_1$ and $D_2$ of $G$ with $D_1 \subset D_2$, $D_1$ has finite index in $D_2$ if and only if $\lv(D_1) = \lv(D_2)$.
\end{definition}

\begin{examples}
(1) An $n$-dimensional local field $F$ is properly suspended, as is the group $GL_m(F)$ for any $m \ge 1$.

(2) Consider $G=\mbb{Q}_p$ levelled over $X= \{ 0 \}$ with elevation $e=1$ as follows. Since $X$ has only one basic open set, distinguished sets are completely determined by the level $n \in \mbb{Z}$. Since the usual topology on $G$ has a countable basis at $0$ given by $\{ p^n \mbb{Z}_p : n \in \mbb{Z} \}$, we can define $G_n = G_{ \{ 0 \}, n} = p^n \mbb{Z}_p$, so that the level topology coincides with the $p$-adic topology. Since $|G_n : G_m| = p^{n-m} < \infty$ for $m \le n$, $G$ is not properly suspended with this level structure.
\end{examples}

Unfortunately, being properly suspended and of ordered type isn't quite strong enough to imply strong statements about levels, as we additionally require some kind of ``completeness" property in order to emulate the case of two-dimensional local fields. However, with additional assumptions, we do obtain such statements, such as the following.

\begin{proposition}
Let $G$ be properly suspended and of ordered type, and let $D \subset G$ be distinguished. Then every open cover $D= \bigcup_i D_i$ of $D$ by distinguished sets of level $\lv(D)$ which has a minimal element, in the sense that either $D \subset D_i$ for all $i$ or there is some $j$ with $|D:D_j| \ge |D:D_i|$ for all $i \neq j$, has a finite subcover.
\end{proposition}

\begin{proof}
Let $D$ be a distinguished set of level $\gamma$, and take any open cover $D \subset \bigcup_i D_i$ by distinguished sets $D_i$ of uniform level $\gamma$. We may assume that $D_i \cap D \neq \emptyset$ for all $i$, since any set not meeting $D$ can be removed from the cover. If any $D_i \supset D$ then $D_i$ is already a finite subcover of $D$. If no $D_i \supset D$, then all $D_i \subset D$ since $G$ is of ordered type.

Let $D_j$ be a minimal element of the open cover. Since we assume that $G$ is properly suspended, there are finitely many $g_k \in G$ such that $D \subset \bigcup_k g_k D_j$. By taking a minimal such cover, we may assume that each $g_k D_j$ contains an element of $D$ which is not contained in any of the others, and hence the $g_k D_j$ are all disjoint since $G$ is of ordered type.

As each $g_k D_j$ is also a distinguished set, for all $i \neq j$ we have $g_k D_j \cap D_i \in \{ g_k D_j, D_i, \emptyset \}$. Since the $g_k D_j$ cover $D$, for each $i$ there is at least one $k$ such that this intersection is nonempty. Furthermore, we can never have $D_i$ properly contained inside $g_k D_j$, since otherwise we would have $|D:D_i| > |D:D_j|$ which contradicts our assumption. 

We thus have the following: for each $i \neq j$ there exists some $k$ such that $g_k D_j \subset D_i$. By Corollary \ref{strictindextowerlaw}, $|D_i : D_j|$ is finite, and so each $D_i$ can be covered by finitely many translates of $D_j$: $D_i = \bigcup_k g_{i,k} D_j$. We thus have $D = \bigcup_{i,k} g_{i,k} D_j$. Since $D_j$ has finite index in $D$, this cover has a finite subcover by $\{ g_{i,k} D_j : i \in I \}$ with $I$ a finite set. In particular, this implies that $D$ is covered by $\{ D_i : i \in I \} \cup \{ D_j \}$, which is a finite subcover of our original cover $D=\bigcup_i D_i$.
\end{proof}

\textbf{For the remainder of this text, we work with a group $G$ levelled over $X$ that is properly suspended and of ordered type.}

\begin{proposition} \label{onefiniteindexgen}
If $D \subset G$ is a distinguished set such that $D = \bigcup_{r=1}^n D_r$ is a disjoint union of finitely many distinguished sets $D_r$ then at least one of the indices $|D:D_r|$ is finite.
\end{proposition}

\begin{proof}
First of all, note that we may order the $D_r$ as follows. If there is some $g \in G$ such that $gD_x \subset D_y$ then we say $D_x \le D_y$. Since $G$ is of ordered type, this is a total ordering on the set of all $D_r$, and by relabelling if necessary we may assume that $D_1 \le D_2 \le \dots \le D_n$. For $1 \le i <n$, taking $g_i \in G$ such that $D_i \subset g_i D_n$ and letting $g_n =1$ gives $\bigcup_{r=1}^n g_r D_n \supset \bigcup_{r=1}^n D_r = D$, which shows that $|D:D_n|$ is finite.
\end{proof}

\begin{proposition} \label{allfiniteindexgen}
If $G$ is of ordered type and $D \subset G$ is a distinguished set such that $D = \bigcup_{r=1}^n D_r$ is a disjoint union of finitely many distinguished sets $D_r$ then the indices $|D:D_r|$ are all finite.
\end{proposition}

\begin{proof}
By Proposition \ref{onefiniteindexgen}, at least one of the indices is finite, and so by relabelling if necessary we may assume that $|D:D_1|$ is finite. Since $G$ is of ordered type, we may take a complete system of representatives $S=\{h_1=I_2, h_2, \dots, h_m \}$, so that $D = \bigcup_i h_i D_1$. We thus have $$D \backslash D_1 = \bigcup_{r=2}^n D_r = \bigcup_{s=2}^m h_s D_1.$$ Taking the intersection with any $D_r$, this gives $$D_r = \bigcup_{s=2}^m \left( D_r \cap h_sD_1 \right).$$

For each $r>1$, let $S_r = \{ h \in S : D_r \cap hD_1 \neq \emptyset \}$. Each $S_r$ is nonempty, since $\bigcup_{h \in S_r} (D_r \cap hD_1) = D_r$. Since $G$ is of ordered type, $D_r \cap hD_1$ is thus equal to either $D_r$ or $hD_1$ for any $h \in S_r$. 

If $D_r \cap hD_1 = D_r$ for any $h$ then we must have $S_r = \{ h \}$ and $D_r \subset hD_1$. But since $D \backslash D_1 = \bigcup_{r=2}^n D_r$, we have $h D_1 = \bigcup_{r \in R} D_r$, where $R$ is a subset of $\{ 2, \dots, n \}$. In particular, we have a distinguished union of disjoint distinguished sets of shorter length, and so by induction each index $|h D_1 : D_r|$ is finite. The tower law for indices then implies that $|D:D_r| = |D:D_1| \cdot |D_1:D_r|$ is finite.

On the other hand, if $D_r \cap hD_1 = hD_1$ for all $h \in S_r$ then we have $h D_1 \subset D_r$. By Corollary \ref{strictindextowerlaw} we thus have $|D:D_1| = |D:hD_1| = |D:D_r| \cdot |D_r : hD_1|$, hence $|D:D_r| \le |D:D_1|$ is finite.
\end{proof}

\begin{corollary}
If $G$ is properly suspended and of ordered type and $D \subset G$ is a distinguished set such that $D = \bigcup_{r=1}^n D_r$ is a disjoint union of finitely many distinguished sets $D_r$ then $\lv(D) = \lv(D_r)$ for all $r$.
\end{corollary}

\begin{proof}
Immediate from the above Proposition and the definition of properly suspended.
\end{proof}

We now generalise the idea of refinements from \cite{waller-GL2} to our current setting. As in the case of $\GL_2(F)$, the following definition is convenient.

\begin{definition}
Let $$A=\bigcup_i \left( \bigcup_j C_{i,j} \backslash \bigcup_k D_{i,k} \right)$$ be a ddd-set. The components $C_{i,j}$ are called the big shells, and the components $D_{i,k}$ are called the small shells.
\end{definition}

We now define refinements in the more compact formulation using this terminology.

\begin{definition}
Let $A \subset G$ be a ddd-set. A refinement of $A$ is a ddd-set $\tilde A$ such that $A = \tilde A$ as sets, every big shell of $A$ is a big shell of $\tilde A$, and every small shell of $A$ is a small shell of $\tilde A$. 
\end{definition}

We also recall the definition of a ddd-set to be reduced.

\begin{definition}
Let $$A=\bigcup_i \left( \bigcup_j C_{i,j} \backslash \bigcup_k D_{i,k} \right)$$ be a ddd-set. We call $A$ reduced if it does not contain any dd-components of the form $B \backslash B$ for a distinguished set $B$.
\end{definition}

As in the particular case of $\GL_2(F)$, the property of being reduced depends the specific presentation of a given ddd-set, and given a non-reduced ddd-set $A$ we may remove all of the superfluous components to obtain a reduced ddd-set $A_{red}$.

Again, we have the following important result concerning refinements.

\begin{theorem} \label{gencommonrefinement}
Let $A$ and $A'$ be reduced ddd-sets with $A=A'$ as sets. Then there exists a reduced ddd-set $\tilde A$ which is a refinement of both $A$ and $A'$.
\end{theorem}

The proof is exactly the same as that of Theorem 3.14 in \cite{waller-GL2}, but with Lemma 3.3 of that text replaced by $G$ being of ordered type. In particular, we have exactly the same algorithm for constructing a refinement of a ddd-set $D$.

\begin{remark}
It should be possible to extend the construction of refinements to groups such as $\mbb{R}(\!(t)\!)$ and $\GL_2\left( \mbb{R} (\!(t)\!) \right)$ by combining the constructions here with classical refinements in the theory of Riemann integration.
\end{remark}

We now turn to the definition of an invariant measure on a properly suspended group $G$ of ordered type. Since $G$ is levelled over a locally compact topological group $X$, in particular $X$ is equipped with a Haar measure $\mu_X$, which is unique up to multiplication by a real constant.

For what follows we will require the following convenient notation. For $\gamma = (\gamma_1, \dots, \gamma_e) \in \mbb{Z}^e$, we will write $Y^\gamma$ to denote the element $Y_1^{\gamma_1} \dots Y_e^{\gamma_e} \in \mbb{R} (\!(Y_1)\!) \dots (\!(Y_e)\!)$. We will also denote the latter group simply by $\mbb{R}(\!(Y)\!)$ for convenience.

As has been done in the specific cases of a two-dimensional local field $F$ (see \cite{fesenko-aoas1}, \cite{fesenko-loop}) and $\GL_2(F)$ (see \cite{waller-GL2}), we wish to define a finitely additive, $\mbb{R} (\!(Y)\!)$-valued left-invariant measure on the family of ddd-sets of $G$. Since the framework of refinements already transfers conveniently to our current situation, this essentially amounts to defining such a measure on distinguished sets.

Let $G_{U, \gamma} \subset G$ be a distinguished set. As far as possible, we would like the measure $\mu$ we will define on $G$ to be compatible with the Haar measure $\mu_X$ on $X$. Since $\mu_X$ is unique only up to a real constant, our $\mu$ will (at best!) also be unique only up to the same scaling. The measure we will define will also be intimately tied to the indexing of particular distinguished sets, and so we will further assume that $G$ is rigid and that the level map $G \mc{L} \rightarrow \mbb{Z}^e$ is surjective.

To begin with, since we want $\mu$ to be left-invariant, we impose the condition $\mu(g G_{U, \gamma}) = \mu( G_{U, \gamma})$ for all $g \in G$. For compatibility with $\mu_X$, we then set $\mu(G_{U, \gamma}) = \mu_X(U) Y^\gamma$. However, we must check that this definition (which essentially comes from the level of the base $X$) agrees with the computation of indices at the level of $G$.

Since $G$ is properly suspended, $G_{U,\gamma} \subset G_{V,\delta}$ will have finite index if and only if $\gamma = \delta$, in which case $U \subset V$ also. Since we seek only a finitely (rather than countably) additive measure, this means that our definition of $\mu$ does not cause any issues between different levels, and so we only have to consider one particular level $\gamma$.

Since $G$ is of ordered type, we have seen previously that, for $G_{U,\gamma} \subset G_{V,\gamma}$, the index $|G_{V, \gamma} : G_{U, \gamma}|$ is finite, and that we may in fact write $G_{V, \gamma}$ as a disjoint union of finitely many (i.e. equal to the index) $G$-translates of $G_{U, \gamma}$. In order for $\mu$ to be finitely additive, we thus require that $$ \mu(G_{V, \gamma}) = |G_{V, \gamma} : G_{U, \gamma}| \mu(G_{U, \gamma}).$$ In terms of our tentative definition of $\mu$, this requires the following definition to hold.

\begin{definition}
Let $G$ levelled over $X$ be properly suspended, rigid, and of ordered type. Then $G$ is called compatible if, for every $U \subset V$, $|G_{V,\gamma} : G_{U,\gamma}|$ is constant as $\gamma$ varies throughout $\mbb{Z}^e$.
\end{definition}

Assuming that $G$ is compatible, we may write the above equality of measures in terms of $\mu_X$: $$\mu_X(V) = |G_{V, \gamma} : G_{U, \gamma}| \mu_X (U).$$ This proves the following.

\begin{proposition}
Suppose $G$ levelled over $X$ is compatible, and that the equality $$\mu_X(V) = |G_{V, \gamma} : G_{U, \gamma}| \mu_X (U)$$ holds for every $U \subset V \in \mc{U}(1)$. Then any left-invariant, finitely additive measure on $G$ must take the form $$\mu : g G_{U, \gamma} \mapsto \mu_X(U) Y^\gamma$$ on distinguished sets.
\end{proposition}

\begin{remark}
We have not yet shown that this map \textit{is} a measure; the above Proposition merely states that any possible measure must be of this form.
\end{remark}

\begin{theorem} \label{otmeasure}
Suppose $G$ levelled over $X$ is compatible, and that the equality $$\mu_X(V) = |G_{V, \gamma} : G_{U, \gamma}| \mu_X (U)$$ holds for every $U \subset V \in \mc{U}(1)$. Let $D$ be a ddd-set, and let $\tilde D$ be a refinement of $D$. The map $$\mu : g G_{U, \gamma} \mapsto \mu_X(U) Y^\gamma,$$ when extended to $\mc{R}$ by additivity, satisfies $\mu(D) = \mu( \tilde D)$, and hence is a well-defined, finitely additive, left-invariant measure compatible with the Haar measure $\mu_X$ on $X$.
\end{theorem}

\begin{proof}
The proof that $\mu$ is well-defined on refinements is again exactly the same as the proof of Proposition 4.2 of \cite{waller-GL2}, since this uses only the definition of ddd-sets and the algorithm of Theorem 3.14 in the same text (which we have in our setting via Theorem \ref{gencommonrefinement}). Theorem \ref{gencommonrefinement} then says that, for any pair $A, A'$ of ddd-sets with $A=A'$ as sets, we have a common refinement $\tilde{A}$, and so we have $\mu(A) = \mu (\tilde{A}) = \mu(A')$, and so $\mu$ is well-defined. The remaining properties then follow immediately from the definitions of $\mu$ and $\mu_X$.
\end{proof}

\begin{remark}
All of the constructions in this section could have equally been used to define instead a right-invariant measure by replacing $g G_{U,\gamma}$ everywhere with $G_{U,\gamma} g$. For the case of $GL_2(F)$, for example, the two measures coincide (this is Corollary 5.7 of \cite{waller-GL2}), but this is not necessarily true for more general $G$ (it is not even true in general for the Haar measure, for instance).
\end{remark}

\begin{example}
Let $F$ be an $n$-dimensional nonarchimedean local field, and let $F_1$ be the $(n-1)^{st}$ residue field, which is a local field. The level structure $F_{t_1^{\gamma_1} \oo_{F_1}, (\gamma_2, \dots, \gamma_n)} = t_1^{\gamma_1} \dots t_n^{\gamma_n} O_F$, where $\oo_{F_1}$ is the ring of integers of $F_1$ and $O_F$ is the rank $n$ ring of integers of $F$, makes $F$ a compatible group. Furthermore, the Haar measure $\mu_{F_1}$ on $F_1$ such that $\mu_{F_1} (\oo_{F_1}) = 1$ satisfies the required formula $\mu_{F_1}(V) = |G_{V, \gamma} : G_{U, \gamma}| \mu_{F_1} (U)$, and so the map $$\mu( \alpha + t_1^{\gamma_1} \dots t_n^{\gamma_n} O_F) = q^{-\gamma_1} Y_2^{\gamma_2} \dots Y_n^{\gamma_n}$$ extends to Fesenko's measure on ddd-sets.
\end{example}

\begin{remark}
We may also apply the results of this section to only a partial level structure. For example, we may consider the subgroups $K_{i,j}=I_2 + t_1^i t_2^j M_2(O_F)$ of $\GL_2(F)$ with $i,j>0$ as forming part of a level structure over $\GL_2(\overline{F})$, with $K_{i,j}$ corresponding to the pair $(K_i, j)$. This does not define a level structure on $\GL_2(F)$ as per our definition since, for example, there is no distinguished set corresponding to $( K_1, -1)$. However as in \cite{waller-GL2} one may define a measure on the ddd-sets generated by the $K_{i,j}$, and this coincides with the measure above wherever both are defined. This suggests that one should be able to extend $\{ K_{i,j} : i,j>0 \}$ to a full level structure for $\GL_2(F)$ over $\GL_2(\overline{F})$.
\end{remark}

\begin{example}
Let $F$ be an $n$-dimensional nonarchimedean local field, and let $G = \GL_m (F)$. For $\gamma_1 \in \mbb{Z}$, $\gamma = (\gamma_2, \dots, \gamma_n) \in \mbb{Z}^{n-1}$, such that $(\gamma_1, \dots, \gamma_n) > (0, \dots, 0)$, we may consider the subgroups $$K_{\gamma_1, \dots, \gamma_n} = I_m + t_1^{\gamma_1} \dots t_n^{\gamma_n} M_m (O_F),$$ where $O_F$ is the rank $n$ ring of integers of $F$. The association $$(K_{\gamma_1}, \gamma) \mapsto K_{\gamma_1, \dots, \gamma_n},$$ equips $G$ with a partial level structure over $\GL_m(F_1)$ of elevation $(n-1)$, where $F_1$ is the $(n-1)^{st}$ residue field of $F$. The map $$\mu(K_{\gamma_1, \dots, \gamma_n}) = \frac{q^{\frac{1}{2}n(n+1)}}{(q^n - 1) (q^{n-1}-1) \dots (q-1)}q^{- n^2 \gamma_1} Y_2^{\gamma_2} \dots Y_n^{\gamma_n} = \mu_{\GL_m(F_1)} (K_m) Y^{\gamma}$$ is a left-invariant, finitely additive measure on the ddd-sets generated by the $K_{\gamma_1, \dots, \gamma_n}$ which satisfies $\mu( \GL_m (O_F)\!) = 1$. Setting $Y_k = X_k^{n^2}$, one obtains in a similar fashion to Theorem 5.5 of \cite{waller-GL2} a nice compatibility between the measure on $G$ and the measure on $F$. More precisely, for $g = (g_{r,s}) \in G$ we have the equality of differentials $$dg = | \det g|_F^{-n} \prod dg_{r,s}. $$
\end{example}

\begin{remark}
The constructions in \cite{waller-GL2} may be able to be modified in order to construct an invariant measure on groups $G$ which are not necessarily compatible. This should agree with the definition in Theorem \ref{otmeasure} in the compatible case, but in general will not be compatible with the Haar measure on $X$.
\end{remark}

\begin{remark}
Consider the following conjecture: any finite union of subsets of type $S$ of a (compatible) group $G$ is also of type $S$. If this conjecture is true, it also makes sense to extend the measure by defining all subsets of type $S$ to have zero volume.
\end{remark}

\section{Induced level structures} \label{sec:induced}
Let $G$ be levelled over $X$, and let $H$ be a subgroup of $G$. Recall from \cite{waller-level} that one may define an induced level structure on $H$ over $X$.

\begin{proposition}
If $G$ has a level structure $\mc{L}$ over $X$, and $H$ is a subgroup of $G$, the collection $\mc{L}_H = \{ H_{U,\gamma} = H \cap G_{U \gamma} : G_{U,\gamma} \in \mc{L} \}$ of subsets of $H$ is a level structure for $H$ over $X$, called the induced level structure.
\end{proposition}

\begin{proof}
This is Lemma 6.2 of \cite{waller-level}.
\end{proof}

In the case of algebraic subgroups, one may in fact consider an induced (partial) level structure over a more appropriate base than the original one.

\begin{proposition}
Suppose $G=\GL_m(F)$ for an $n$-dimensional nonarchimedean field $F$, with the partial level structure over $\GL_m(F_1)$ given by the distinguished subgroups $K_{\gamma_1, \dots, \gamma_n}$. Let $H$ be a subgroup of $G$ defined by finitely many polynomial equations $f_1, \dots, f_k \in O_F[\![X_1, \dots, X_{m^2}]\!]$, i.e. $$ H = \{ (g_{r,s}) \in G : f_i ( (g_{r,s} ) ) = 0, 1 \le i \le k  \}, $$ and let $\overline{H}$ be the subgroup of $\GL_m(F_1)$ defined by the reductions $$\bar f_1, \dots, \bar f_k \in F_1 [\![X_1, \dots, X_{m^2} ]\!].$$ If all of the polynomials $\bar f_i$ are separable (i.e. all the roots are simple), the association $(U \cap \overline{H} , \gamma) \mapsto H \cap G_{U,\gamma}$ defines a (partial) level structure for $H$ over $\overline{H}$ whose distinguished sets coincide with those of the induced level structure.
\end{proposition}

\begin{proof}
We must first check that the association is well-defined. In other words, if $\overline{H} \cap K_{\alpha} = \overline{H} \cap K_{\beta}$ for $\alpha, \beta \in \mbb{Z}$ we must make sure that $H \cap K_{\alpha, \gamma} = H \cap K_{\beta, \gamma}$ for all $\gamma \in \mbb{Z}^{n-1}$. Without loss of generality, we may assume that $\beta \le \alpha$, so that we have $H \cap K_{\alpha, \gamma} \subset H \cap K_{\beta, \gamma}$.

Let $g = (g_{r,s}) \in H \cap K_{\beta, 0}$. Then the reduction $\bar g = (\bar g_{r,s}) \in K_{\beta}$ is a root of all of the polynomials $\bar f_i$, and so we have $\bar g \in \overline{H} \cap K_{\beta} = \overline{H} \cap K_{\alpha}$. By Hensel's Lemma, there exists a unique $g' \in K_{\beta,0}$ with $\bar{g'} = \bar g$ and $f_i(g')=0$, and there exists a unique $g'' \in K_{\alpha,0}$ with $\bar{g'} = \bar g$ and $f_i(g'')= 0$. Since $f_i(g) = 0$, uniqueness of $g'$ forces $g=g'$. Furthermore, since $K_{\alpha, 0} \subset K_{\beta, 0}$, we have $g'' \in K_{\beta,0}$, and so uniqueness of $g'$ forces also $g'=g''$. We thus have $g = g'' \in H \cap K_{\alpha,0}$ and so we have the required equality for $\gamma =0$. The result for general $\gamma$ then follows from the fact that the map $I_m + t_1^{\gamma_1} \dots t_n^{\gamma_n} M \mapsto I_m + t_1^{\gamma_1} M$ is an isomorphism $K_{\gamma_1, \gamma} \rightarrow K_{\gamma_1, 0}$.

The fact that the distinguished sets coincide with those of the induced level structure is then immediate from the definition, and this implies that conditions (1) and (2) in the definition of a level structure hold. It remains to check (3) and (4).

Let $U,V \in \mc{U}(1)$ and $\gamma \le \delta \in \mbb{Z}^{n-1}$. We have $$G_{V \cap \overline{H}, \gamma} \cap G_{U \cap \overline{H}, \delta} = (H \cap G_{V, \gamma} ) \cap (H \cap G_{U, \delta}) = H \cap G_{V, \delta} = G_{V \cap \overline{H}, \delta},$$ and so (3) is satisfied.

Similarly, \begin{align*}
G_{U \cap \overline{H}, \gamma} \cup G_{V \cap \overline{H}, \gamma} &= (H \cap G_{U, \gamma}) \cup (H \cap G_{V, \gamma}) \\
&= H \cap (G_{U, \gamma} \cup G_{V, \gamma}) \\
&= H \cap G_{U \cup V, \gamma} = G_{(U \cup V) \cap \overline{H}, \gamma},\end{align*} and \begin{align*}
G_{U \cap \overline{H}, \gamma} \cap G_{V \cap \overline{H}, \gamma} &= (H \cap G_{U, \gamma}) \cap (H \cap G_{V, \gamma}) \\
&= H \cap (G_{U, \gamma} \cap G_{V, \gamma}) \\
&= H \cap G_{U \cap V, \gamma} \\
&= G_{(U \cap V) \cap \overline{H}, \gamma}, \end{align*} and so (4) is satisfied.
\end{proof}

\begin{remark}
The proof of the above Proposition in fact works for algebraic groups over any Henselian ring $R$ whose residue field is locally compact.
\end{remark}

In particular, we have the following.

\begin{theorem}
Let $G$ be an algebraic subgroup of $\GL_m(F)$ defined by the equations $f_1, \dots f_k$, and let $\overline{G}$ be the subgroup of $\GL_m(F_1)$ defined by the reductions $\bar f_1, \dots, \bar f_k$. If $G$ is properly suspended and rigid over $\overline{G}$ with respect to the induced (partial) level structure, if $|K_{i, \gamma} \cap G   : K_{j, \gamma} \cap G | = |K_{i} \cap \overline{G} : K_{j} \cap \overline{G}| $ for all $0<i \le j$ and all $\gamma > 0 \in \mbb{Z}^{n-1}$, and if the polynomials $\bar f_1, \dots, \bar f_k$ have no multiple roots, there is a left-invariant, finitely additive measure on $G$ given by $$\mu(K_{i, \gamma} \cap G) = \bar \mu (K_{i} \cap \overline{G}) Y^\gamma,$$
where $\bar \mu$ is the Haar measure on $\overline{G}$.
\end{theorem}

\begin{proof}
The equality $|K_{i, \gamma} \cap G   : K_{j, \gamma} \cap G | = |K_{i} \cap \overline{G} : K_{j} \cap \overline{G}| $ is exactly the one that is required for the measure to exist in the previous section under the conditions of being properly suspended, rigid (which are both true for $G$ by assumption), and of ordered type (which is true of $G$ as a subgroup of a group of ordered type).
\end{proof}

Once we already know examples where this Theorem holds, one may then use the following Lemma to easily produce further examples.

\begin{lemma} \label{shortexactmeasure}
Let $$1 \rightarrow N \rightarrow G \rightarrow Q \rightarrow 1$$ be a short exact sequence of algebraic subgroups of $\GL_m(F)$. If any two of $N$, $G$, $Q$ satisfy the conditions of Theorem \ref{alggroupmeasure}, then so does the third.
\end{lemma}

\begin{proof}
Write ${N_{i, \gamma}} = K_{i, \gamma} \cap N$, ${G_{i, \gamma}} = K_{i, \gamma} \cap G$, ${Q_{i, \gamma}} = K_{i, \gamma} \cap Q$, and similarly write $N_i = K_i \cap N$, $G_i = K_i \cap G$, $Q_i = K_i \cap Q$. For $i < \ell$, applying the snake lemma to the commutative diagrams
\begin{displaymath}
\begin{tikzcd}
1 \arrow[r] & {N_{\ell}} \arrow[r] \arrow[d] & {G_{\ell}} \arrow[r] \arrow[d] & {Q_{\ell}} \arrow[r] \arrow[d] & 1 \\
1 \arrow[r] & {N_{i}} \arrow[r] & {G_{i}} \arrow[r] & {Q_{i}} \arrow[r] & 1 
\end{tikzcd}
\end{displaymath}
and
\begin{displaymath}
\begin{tikzcd}
1 \arrow[r] & {N_{\ell,\gamma}} \arrow[r] \arrow[d] & {G_{\ell,\gamma}} \arrow[r] \arrow[d] & {Q_{\ell,\gamma}} \arrow[r] \arrow[d] & 1 \\
1 \arrow[r] & {N_{i,\gamma}} \arrow[r] & {G_{i,\gamma}} \arrow[r] & {Q_{i,\gamma}} \arrow[r] & 1 
\end{tikzcd}
\end{displaymath}
gives the short exact sequences of cokernels 
\begin{equation*}
1 \rightarrow {N_{i}} / {N_{\ell}} \rightarrow {G_{i}} / {G_{\ell}} \rightarrow {Q_{i}} / {Q_{\ell}} \rightarrow 1,
\end{equation*}
\begin{equation*}
1 \rightarrow {N_{i,\gamma}} / {N_{\ell,\gamma}} \rightarrow {G_{i,\gamma}} / {G_{\ell,\gamma}} \rightarrow {Q_{i,\gamma}} / {Q_{\ell,\gamma}} \rightarrow 1.
\end{equation*}
In particular, we have the equalities of indices
$$ |G_i : G_\ell | = |N_i : N_\ell| \cdot |Q_i : Q_\ell|$$
$$ |G_{i, \gamma} : G_{\ell, \gamma} | = |N_{i, \gamma} : N_{\ell, \gamma}| \cdot |Q_{i,\gamma} : Q_{\ell,\gamma}|,$$
and so if any two of $N, G, Q$ satisfy the required index equalities of Theorem \ref{alggroupmeasure} then so must the third.
\end{proof}

As a first example, we obtain an invariant measure on the group $SL_m(F)$.

\begin{theorem}
Let $F$ be a two-dimensional local field and let $G=SL_m(F)$. For $(i,j) > (0,0)$ let $\widetilde{K_{i,j}} = K_{i,j} \cap G$. The map $$\mu(g \widetilde{K_{i,j}}) = \lambda q^{-(m^2 -1)i}X^{(m^2-1)j}$$ for $\lambda \in \mbb{R}(\!(X)\!)^\times$ is a left-invariant, finitely additive measure on $G$.
\end{theorem}

\begin{proof}
Here the required polynomial is simply $\det - 1 $, and we may apply Lemma \ref{shortexactmeasure} to the short exact sequence $$1 \rightarrow F^\times \rightarrow \GL_m(F) \rightarrow SL_m(F) \rightarrow 1 $$ to deduce the existence of an invariant measure $\mu_{SL_m}$ on $SL_m(F)$. To see that it may be written in the form $\mu_{SL_m}(g \widetilde{K_{i,j}}) = \lambda q^{-(m^2 -1)i}X^{(m^2-1)j}$, note that the proof of Lemma \ref{shortexactmeasure} implies that we may organise so that $$\mu_Q(D) = \frac{ \mu_G(D)}{\mu_N(D)}$$ for any distinguished set $D$.
\end{proof}

\begin{example}
Recall that the groups $\textbf{A}_S$ and $\mbb{A}_{S_{'}}$ of geometric and analytic adeles associated to a surface $S$ arise as restricted products of two-dimensional local fields. Since the functor $\GL_2 (\cdot)$ behaves well with respect to (restricted) products, we may consider the groups $\GL_2( \textbf{A}_S)$ and $\GL_2(\mbb{A}_{S_{'}})$, along with their algebraic subgroups with the induced level structure. Since the formula $$\mu_X(V) = |G_{V, \gamma} : G_{U, \gamma}| \mu_X (U)$$ holds at each local component, and since all but finitely many will be equal to $1$, it in fact holds globally. There thus exists a left-invariant measure on any properly suspended, rigid adelic algebraic group $G$ defined by polynomials with separable reduction.
\end{example}

\begin{remark}
In \cite{fesenko-aoas2}, Fesenko shows that it is not possible to define a measure on the geometric adeles $\textbf{A}_S$ which is both compatible with the one-dimensional theory (for example, in the sense of the above example) \textit{and} satisfies the important property $\mu(\oo \textbf{A}_S)=1$ for an appropriate integral structure $\oo \textbf{A}_S$. Indeed, this is one of the major motivations for the introduction of the analytic adelic structures.
\end{remark}

\begin{remark}
We have seen that the existence of an invariant measure on a group $G$ which is compatible with the measure on $X$ heavily depends on the choice of the base space $X$ (as one would expect!). However, once defined, the measure can be seen to be intrinsic to $G$; if we afterwards consider $G$ with the same level structure but over a different base $X'$, the measure on $G$ still exists, it just may not have anything to do with the space $X'$. (Consider, for example, $\GL_2(F)$ levelled over $\GL_2 (\overline{F})$ versus $\GL_2(F)$ levelled over $F$.)

This seems to be a very common theme concerning groups with level structure. While the topological and analytic properties of $G$ should be essentially self-contained (since the information is bound to the level structure, which is in principle just a collection of subsets of $G$), by choosing the correct base $X$ for the level structure we may see various analogies between $G$ and $X$ which makes certain properties of $G$ appear more clearly.
\end{remark}

\begin{remark}
It was noticed by Morrow (\cite{morrow-fubini}) that Fubini's Theorem does not always hold for multiple integrals over a higher dimensional local field. Since this is just a particular example of our general constructions, it follows that the integral with respect to the Fesenko measure on a general $G$ does not satisfy Fubini's Theorem. One may try to generalise \cite{morrow-fubini} to investigate what Fubini type properties may hold in general.
\end{remark}

\end{document}